\documentclass[12pt]{amsart}
\usepackage{amsmath,amsthm,amsfonts,amssymb,latexsym,enumerate,url,hyperref} 
\usepackage{color, xcolor,here}

\begin{document}

\theoremstyle{plain}

\newtheorem{thm}{Theorem}[section]
\newtheorem{lem}[thm]{Lemma}
\newtheorem{conj}[thm]{Conjecture}
\newtheorem{pro}[thm]{Proposition}
\newtheorem{cor}[thm]{Corollary}
\newtheorem{que}[thm]{Question}
\newtheorem{rem}[thm]{Remark}
\newtheorem{defi}[thm]{Definition}
\newtheorem{hyp}[thm]{Hypothesis}

\newtheorem*{thmA}{THEOREM A}
\newtheorem*{thmB}{THEOREM B}
\newtheorem*{conjA}{CONJECTURE A}
\newtheorem*{conjB}{CONJECTURE B}

\newtheorem*{thmAcl}{Main Theorem$^{*}$}
\newtheorem*{thmBcl}{Theorem B$^{*}$}

\def\irrp#1{{\rm Irr}_{p'}(#1)}

\def\Z{{\mathbb Z}}
\def\C{{\mathbb C}}
\def\Q{{\mathbb Q}}
\def\irr#1{{\rm Irr}(#1)}
\def\ibr#1{{\rm IBr}(#1)}
\def\irra#1{{\rm Irr}_{\rm A}(#1)}
\def\ibra#1{{\rm IBr}_{\rm A}(#1)}
\def \c#1{{\cal #1}}
\def\cent#1#2{{\bf C}_{#1}(#2)}
\def\syl#1#2{{\rm Syl}_#1(#2)}
\def\nor{\trianglelefteq\,}
\def\oh#1#2{{\bf O}_{#1}(#2)}
\def\Oh#1#2{{\bf O}^{#1}(#2)}
\def\zent#1{{\bf Z}(#1)}
\def\det#1{{\rm det}(#1)}
\def\ker#1{{\rm ker}(#1)}
\def\norm#1#2{{\bf N}_{#1}(#2)}
\def\alt#1{{\rm Alt}(#1)}
\def\iitem#1{\goodbreak\par\noindent{\bf #1}}
   \def \mod#1{\, {\rm mod} \, #1 \, }
\def\sbs{\subseteq}

\def\gc{{\bf GC}}
\def\ngc{{non-{\bf GC}}}
\def\ngcs{{non-{\bf GC}$^*$}}
\newcommand{\notd}{{\!\not{|}}}
 \renewcommand{\thefootnote}{\fnsymbol{footnote}}
\footnotesep6.5pt

\title{Character Values of $p$-solvable groups on  picky elements}

\author{Alexander Moret\'o}
\address{Departament de Matem\`atiques, Universitat de Val\`encia, 46100 Burjassot,
Val\`encia, Spain}
\email{alexander.moreto@uv.es}

\author{Gabriel Navarro}
\address{Departament de Matem\`atiques, Universitat de Val\`encia, 46100 Burjassot,
Val\`encia, Spain}
\email{gabriel@uv.es}

\author{Noelia Rizo}
\address{Departament de Matem\`atiques, Universitat de Val\`encia, 46100 Burjassot,
Val\`encia, Spain}
\email{noelia.rizo@uv.es}

\thanks{This research is partially supported by Ministerio de Ciencia e Innovaci\'on (Grant PID2022-137612NB-I00 funded by MCIN/AEI/10.13039/501100011033 and ``ERDF A way of making Europe")   and CIDEIG/2022/29 funded by Generalitat Valenciana.}

\keywords{Character correspondences,  $p$-solvable group, picky $p$-elements}

\subjclass[2010]{Primary 20C15}

\begin{abstract}
A new conjecture on characters of finite groups, related to the McKay conjecture, was proposed recently by the first and third authors. In this paper, we prove it for $p$-solvable groups when $p$ is odd. 
 \end{abstract}

\maketitle

\section{Introduction}  

A new conjecture on characters of finite groups  was proposed recently by the first and the third authors \cite{MR}.
The simplest form of this conjecture asserts the following:
if $G$ is a finite group, $p$ is a prime, and $x \in G$ is a $p$-element which lies in a unique
Sylow $p$-subgroup $P$ of $G$ (we say that $x$ is {\it picky} in $G$), then there is a bijection ${\rm Irr}^x(G) \rightarrow {\rm Irr}^x(\norm GP)$
satisfying certain strong conditions, where ${\rm Irr}^x(G)$ is the set of the irreducible complex characters $\chi$ of $G$ 
such that $\chi(x) \ne 0$. These {\it picky} elements are shown to be fairly abundant in \cite{MMM}. The purpose of this paper is to prove the following strong form of the conjecture for $p$-solvable groups for odd primes. Namely:

\begin{thmA}
Let $G$ be a finite $p$-solvable group, where $p$ is an odd prime, and let $x \in G$ be a $p$-element which lies in a unique
Sylow $p$-subgroup $P$ of $G$. Then  there is a bijection $^*: {\rm Irr}^x(G) \rightarrow {\rm Irr}^x(\norm GP)$
and a sign $\epsilon_x$ such that $\chi(x)=\epsilon_x \chi^*(x)$ and $\chi(1)_p=\chi^*(1)_p$ for all 
$\chi \in {\rm Irr}^x(G)$.
\end{thmA}

In order to prove this theorem we need to use  deep results from the character theory of finite groups, like the Dade--Turull character correspondence (see Section 8.3 of \cite{N18}, for instance).   
Unfortunately, our methods are not sufficient to handle the case $p=2$. As we briefly explain in the final section of the paper, this limitation is related, among other things,  to the behavior of the Weil character associated with symplectic groups when evaluated on $2$-elements of order at least $4$.

\medskip
We refer the reader to \cite{M, Ma} for progress on the conjecture for group of Lie type and symmetric groups, respectively.

\section{Preliminaries}

We begin by recalling the following well-known result.

\begin{lem}\label{tom}
Suppose that $A$ acts as automorphism on a finite group $G$.
Let $N\nor G$ be $A$-invariant, with $(|G:N|, |A|)=1$. Let $C/N=\cent{G/N}A$.
\begin{enumerate}[(a)]
\item
If $\chi \in \irr G$ is $A$-invariant, then $\chi_N$ has an $A$-invariant irreducible
constituent $\theta$ and the set of the $A$-invariant irreducible constituents 
of $\chi_N$ is $\{\theta^c \, |\, c \in C\}$.  Hence,  $\theta$ is unique if $C=N$.
\item
If $\theta \in \irr N$ is $A$-invariant, then $\theta^G$ has an $A$-invariant irreducible
constituent $\chi$. Furthermore, $\chi$ is unique if $C=N$.
\item
Let $\chi \in \irr G$ and let $\theta \in \irr N$ under $\chi$.
If $C=G$, then $\chi$ is $A$-invariant if and only if $\theta$ is $A$-invariant.
\end{enumerate}
\end{lem}
\begin{proof}
For (a), use Theorem 13.27 of \cite{Is76}, using Corollary 13.9 of \cite{Is76}.
Notice that if $\theta$ is $A$-invariant and $c \in C$, then $\theta^c$ is also $A$-invariant.
For (b), see Theorem 13.31 and Problem 13.10 of \cite{Is76}.
For (c), see Corollary 1.4 of \cite{W79}.
\end{proof}

An important tool in our proofs is the relative Glauberman correspondence.
If $A$ acts by automorphisms on a finite group $G$, let ${\rm Irr}_A(G)$ be
the set of $A$-invariant irreducible characters of $G$.

\begin{thm}\label{rel}
Suppose that $P$ is a $p$-group acting on a finite group $G$ by automorphisms.
Suppose that $N \nor G$ is $P$-invariant and $G/N$ has order not divisible by $p$.
Let $C/N=\cent{G/N}P$.
Then there is a natural bijection $^*:{\rm Irr}_P(G) \rightarrow {\rm Irr}_P(C)$.
In fact, if $\chi \in {\rm Irr}_P(G)$, we have that 
$$\chi_C=e\chi^* + p \Delta + \Xi \, ,$$
where $e \equiv \pm 1$ mod $p$,  $\Delta$ and $\Xi$ are characters of $C$ (or zero),
and no irreducible constituent of $\Xi$ lies over some $P$-invariant character of $N$.
Furthermore, if $\tau \in {\rm Irr}_P(C)$, then we can write
$$\tau^G=d\mu + p \Psi + \rho \, ,$$
where $\mu \in {\rm Irr}_P(G)$  is such that $\mu^*=\tau$, $d \equiv \pm 1$ mod $p$, 
$\Psi$ and $\rho$ are characters of $G$ (or zero), and no irreducible constituent of
$\rho$ is $P$-invariant.
\end{thm}

\begin{proof}
This is Theorem E of \cite{NTV}, where we have added a second part which we prove now.
Notice that $[\chi_C, \chi^*]$ is not divisible by $p$, since $[\chi^*, \Xi]=0$ (using Lemma \ref{tom}(c)).
Let $\tau \in {\rm Irr}_P(C)$ and let $\mu \in {\rm Irr}_P(G)$ such that $\mu^*=\tau$.
Hence we can write $$\tau^G=f\mu + p \Psi + \rho \, ,$$
where $\Psi$ and $\rho$ are characters of $G$ (or zero),   
$f=[\tau^G, \mu]$ is not divisible by $p$,  $[\rho, \mu]=0$, and
all irreducible constituents of $\rho$ occur with multiplicity not divisible by $p$.
Suppose that $\xi \in \irr G$ is $P$-invariant and is a constituent of $\rho$.
Then $[\tau^G, \xi]=[\xi_C, \tau]$ is not divisible by $p$.
Write $\xi_C=q\xi^* + pT + S$, where no irreducible constituent of $S$ lies over some irreducible
$P$-invariant of $N$. Since $\xi \ne \mu$, then $\xi^* \ne \mu^*=\tau$. Since $\tau$ occurs with multiplicity not divisible
by $p$, necessarily, we have that $[S, \tau]\ne 0$. But, since $\tau$ is $P$-invariant and $|C/N|$ is coprime to $p$, we have that $\tau_N$ has $P$-invariant irreducible constituents, and this is
a contradiction.
\end{proof}

We will call the character $\varphi$ above the {\it relative $A$-Glauberman} correspondent of $\theta$. 

\begin{lem}\label{inv2}
Suppose that $G$ is a finite group and that $L\leq K\leq G$ are normal subgroups of $G$ such that
$K/L$ is a $p'$-group. Suppose that $PK$ is normal in $G$, where $P \in \syl pG$.
Let $H=L\norm GP$ and write $K\cap H=C$.  Let $xL \in H/L$ be a $p$-element such that $\cent{K/L}{x}=C/L$.
Let $\theta \in \irr K$ be $x$-invariant, and let $\varphi \in \irr C$ be the relative $\langle x\rangle$-Glauberman correspondent
of $\theta$.   Then $H_\theta=H_\varphi$. 
\end{lem}

\begin{proof}
Let $U=PL$. By Theorem \ref{rel} we can write $$\theta_C=e\varphi + p \Delta + \Xi, $$
where $e \equiv \pm 1$ mod $p$ and no irreducible constituent of $\Xi$ lies over an $x$-invariant irreducible
character of $L$.  By Lemma \ref{tom}, let $\xi \in \irr L$ be $x$-invariant under $\varphi$.
Now let $h \in P$ and let $R=\langle h, x\rangle$. 
Suppose that $\theta$ is $h$-invariant. Since $\cent{K/L}P=\cent{K/L} x=C/L$,
we have that $C/L=\cent{K/L} R$. By Lemma \ref{tom}(a), let $\xi_1 \in \irr L$ be $R$-invariant under $\theta$.
By Lemma \ref{tom}(a) applied to the $x$-action, we conclude that $\xi_1=\xi^c$ and $\xi$ is also $R$-invariant.
Now, $\varphi^h$ lies over $\xi^h=\xi$ and under $\theta^h=\theta$. Thus $\varphi^h$ is $x$-invariant by
Lemma \ref{tom}(c). Necessarily, we have that $\varphi^h=\varphi$.

Conversely, suppose that $\varphi^h=\varphi$. Thus $\varphi$ is $R$-invariant.
Write $\varphi^K=e\theta + p\Psi + \rho$,
where no irreducible constituent of $\rho$ is $x$-invariant. 
Now, if $\tilde\theta \in {\rm Irr}_R(K)$ is the $R$-relative Glauberman correspondent of $\varphi$,
we have that $[\varphi^K, \tilde\theta]$ is not divisible by $p$. We conclude that $\tilde\theta=\theta$
is $R$-invariant. Thus we have proved that $P_\theta=P_\varphi$.
Hence, $U_\theta=U_\varphi$.  Notice that $U_\theta \nor H_\theta$ since $U\nor H$.

Let $y \in H_\theta$. Then $\theta^y=\theta$.  Hence $x^{y^{-1}} \in U_\theta=U_\varphi$.  Now, $\varphi^y$ is an irreducible
constituent of $\theta_C$, with multiplicity not divisible by $p$. But also $\varphi^{yx}=\varphi^{y}$ because $yxy^{-1}$ fixes $\varphi$.
By uniqueness, $\varphi^{y}=\varphi$, since $\varphi^y$ lies over some $x$-invariant character of $L$. Thus $H_\theta \sbs H_\varphi$. The same argument going up
shows the equality.
\end{proof}

\begin{lem}\label{elem2}
Suppose that $G$ is a finite group and that $L\leq K\leq G$ are normal subgroups of $G$ such that
$K/L$ is a $p'$-group. Suppose that $H \le G$ is such that $G=KH$ and write $K\cap H=C$.
Let $x\in H$ such that $xL \in H/L$ is a $p$-element and suppose that $H/L$ has a normal Sylow $p$-subgroup. Then $x^g \in H$ implies $g\in H$ for any $g\in G$ if and only if $\cent{K/L}x\subseteq C/L$. Moreover, if this happens then $\mu^G(x)=\mu(x)$ for all characters $\mu$ of $H$.
\end{lem}
\begin{proof}
Let $k \in K$. Notice that  $k^xL=kL$ if and only if $x^kL=xL$. 
Suppose that $x^g \in H$ implies $g\in H$ for any $g \in G$. Let $kL \in K/L$ centralized by $x$.
Then $x^k \in H$ and thus $k \in H \cap K=C$, by hypothesis.

Conversely, assume that $\cent{K/L} x\subseteq C/L$. Since $C$ is normal in $H$, then $\cent{K/L}{x^h}=\cent{K/L} x^h\subseteq C/L$ for all $h \in H$.
Now, suppose that $x^k \in H$ for some $k \in K$. Then $x^{-1} x^k \in H$. But $x^{-1}k^{-1} x k \in K$.
Thus $x^{-1} x^k \in C$. Therefore $x^{-1}x^kL\in C/L$ and it is a $p$-element (because $x^{-1}L,x^kL$ lie in the unique Sylow $p$-subgroup of $H/L$). Since $C/L$ is a $p'$-group, we conclude that $x^{-1}x^k\in L$, and thus $k^xL=kL$. Then $k \in \cent {K/L}x\subseteq C \sbs H$. 
Finally, if $x^{hk} \in H$ for some $h\in H$ and $k\in K$, then using that $\cent{K/L} {x^h}\subseteq C$, we have that $k \in H$ and we are done.
 
 For the last assertion, if $\mu$ is a character of $H$, then
$$\mu^G(x)={1 \over |H|} \sum_{g \in G \atop g^{-1}xg \in H} \mu(g^{-1}xg)=
{1 \over |H|} \sum_{g \in H} \mu(g^{-1}xg)=\mu(x) \, ,$$
as desired. 
\end{proof}

Although we shall not need it, the following elementary fact might perhaps be useful in the future.

\begin{pro}
Let $G$ be a finite group, let $H\leq G$ and let $x\in H$ be a $p$-element. Then $(1_H)^G(x)=1$ if and only if
$H$ contains all the $p$-Sylow normalizers that contain $x$.
\end{pro}

\begin{proof}
Suppose that $H$ contains all the $p$-Sylow normalizers that contain $x$.
In particular, $|G:H|$ is $p'$.  Suppose that $x \in P \in \syl pH$.
Let $g \in G$ be such that $x^g \in H$.  Then $x^g \in P^h$ for some $h \in H$.
Thus $x \in P^{hg^{-1}}$. Hence, $P^{hg^{-1}} \sbs H$, by hypothesis.
Therefore, $P^{hg^{-1}}=P^{h_1}$ for some $h_1 \in H$. We have that $hg^{-1}h_1^{-1}$ normalizes $P$,
and therefore is contained in $H$. We conclude that $g \in H$. Therefore 
$(1_H)^G(x)=1$, as wanted.

Conversely, suppose that $(1_H)^G(x)=1$. Then, if $x^g \in H$ for some $g \in G$ we have that $g\in H$.
Let $Q$ be a Sylow $p$-subgroup of $H$ containing $x$.
Let $P \in \syl pG$ such that $P\cap H=Q$.  If $Q<P$, then $Q< \norm PQ$. Let $y \in \norm PQ-Q$.
Then $x^y \in Q$ but $y \not\in H$, a contradiction. So we may assume that $Q\in \syl pG$.
Suppose that $x \in R \in \syl pG$. Let $y \in \norm GR$.
Then $R^g=Q$ for some $g \in G$. Then $x^g \in H$ and therefore
$g \in H$.  Also, $x^{yg} \in H$, and thus $yg \in H$. Hence $y \in H$, as wanted.
\end{proof}

Note that by Proposition 2.6 of \cite{M} this means that $(1_H)^G(x)=1$ if and only if $H$ contains the subnormalizer of $x$.

\smallskip

\section{The Key result}

We start by proving the following:

\begin{thm}\label{fullycentralcase}
Suppose that $G$ is a finite group, $p$ is an odd prime, $L\leq K\leq G$ are normal subgroups of $G$ such that
$K/L$ is a $p'$-group and $L\subseteq{\rm \textbf{Z}}(G)$. Suppose that $V=PK$ is normal in $G$, where $P \in \syl pG$.
Let $H=\norm GP$. Note that $G=KH$ and write $K\cap H=C$. 
Let $\theta \in \irr K$ be $G$-invariant, and let $\varphi \in \irr C$ be the relative $P$-Glauberman correspondent
of $\theta$.  Then there exists a bijection
$$f: {\rm Irr}(G|\theta) \rightarrow {\rm Irr}(H|\varphi)$$ such that
$\chi(1)_p=f(\chi)(1)_p$ for every $\chi \in \irr{G|\theta}$. Furthermore,   if $x \in P$ is such that $xL$ is picky in $G/L$, then  
$\chi(x)=\varepsilon f(\chi)(x)$, where $\varepsilon\in\{1,-1\}$ is the unique sign such that
 $[\theta_C,\varphi]\equiv \varepsilon$ mod $p$.
 \end{thm}
 
 \begin{proof}
 Since $L$ is central in $G$, $K$ has a central Sylow $p$-subgroup, say $L_p$, and
it follows that $K=X \times L_p$ for a unique $X\nor G$.
It follows that $\theta=\theta_{p'} \times \theta_p$, and $\varphi=\varphi_{p'} \times \theta_p$.
Notice that $\cent XP=C\cap X$ and that $\varphi_{p'}$ is the $P$-Glauberman correspondent of $\theta_{p'}$.
Notice that $\theta_{p'}$ has a canonical extension $\widehat{\theta_{p'}}$
to $V$, and $\varphi_{p'}$ has a canonical extension $\widehat{\varphi_{p'}}= \varphi_{p'} \times 1_P$ to $U=V\cap H$.
It is well-known that there is a character triple isomorphism
$(G, X, \theta_{p'}) \rightarrow (H, L_{p'}, \varphi_{p'})$ that sends $\widehat{\theta_{p'}}$
to $\widehat{\varphi_{p'}}$ (see the proof of Lemma 7.3 of \cite{T08}).  Let $$^*: \irr{G|\theta_{p'}} \rightarrow \irr{H|\varphi_{p'}}$$
be the bijection associated to the character triple isomorphism. Let $\chi \in \irr{G|\theta}$. By Gallagher's correspondence, we can write $\chi_V=\delta \widehat{\theta_{p'}}$,
where $\delta$ is a character of $V/X$ all of whose irreducible constituents lie over $\theta_p$.
By the character triple isomorphism, we have that 
$(\chi^*)_U=\delta \widehat{\varphi_{p'}}$, where we identify $\delta$ with the corresponding character
via the natural isomorphism $V/X \rightarrow U/L_{p'}$. Hence $\chi^*$ lies over $\theta_p$ and hence the restriction of ${}^*$ to $\irr{G|\theta}$ gives a bijection

$${}^*: \irr{G|\theta} \rightarrow \irr{H|\varphi}.$$ We next prove that this is the desired bijection. First since character triple isomorphisms preserve ratios of character degrees we have that $\chi(1)_p=\chi^*(1)_p$. Now, let $Y=K\langle x \rangle$, and notice that $Y/X$ is a $p$-group. Let $\tau\in{\rm Irr}(Y)$ be the canonical extension of $\theta_{p'}$ to $Y$. Now if $\chi\in{\rm Irr}(G|\theta)$ we have that $\chi_Y=\xi\tau$ for some character $\xi$ of $Y/X$ and then $\chi(x)=\xi(x)\tau(x)$. Our goal now is to show that $\chi^*(x)=\xi(x)\tau^*(x)$. By the proof of Lemma 5.17 of \cite{N18}, we have that $\chi^*(x^*)=\tau^*(x^*)\xi^*(x^*)=\tau^*(x^*)\xi(x)$. Since the isomorphism of character triples is associated to the natural map $gX\rightarrow gL_{p'}$ we have that $x^*=x$ and hence $\chi^*(x)=\xi^*(x)\tau^*(x)=\xi(x)\tau^*(x)$, as wanted.
Now, we consider the action of $S=\langle x \rangle$ on $X$. Since $xL$ is picky in $G/L$ we have that $$\cent X x/L\subseteq \cent {X/L}{xL}\subseteq\norm {X/L}{PL/L}=(C\cap X)/L=\cent X P/L$$ and hence $\cent X S=\cent X P$. Therefore the $S$-Glauberman correspondent of $\theta_{p'}$ is also $\varphi_{p'}$. Now, by Theorem 13.6 of \cite{Is76} applied to the action of $S=\langle x \rangle$ on $X$, we have that $\tau(x)=\varepsilon \varphi_{p'}(1)$ where $\varepsilon\in\{1,-1\}$. Also, notice that, since $L\langle x\rangle/L_{p'}$ is a $p$-group and $L_{p'}$ is central, we have that $L\langle x\rangle=L_{p'}\times W$ for $W\in{\rm Syl}_p(L\langle x\rangle)$. Now $L\langle x \rangle/L$ is normal in $C\langle x \rangle/L$ and hence $W$ is normal in $C\langle x \rangle$. Therefore $C\langle x\rangle=(X\cap C)\times W$ and since $\tau^*$ is the canonical extension, we have that $\tau^*=\varphi_{p'}\times 1_W$ and then  $\tau^*(x)=\varphi_{p'}(1)$. Then

 $$\chi(x)=\varepsilon\varphi_{p'}(1)\xi(x)=\varepsilon\chi^*(x).$$

 To finish the proof, we need to show that $\varepsilon$ is the unique sign such that $[\theta_C,\varphi]\equiv \varepsilon$ mod $p$. By Theorem 13.14 of \cite{Is76} we have that $\varepsilon\equiv [(\theta_{p'})_{C\cap X},\varphi_{p'}]$ mod $p$. Write $[(\theta_{p'})_{C\cap X},\varphi_{p'}]=f$. Following Theorem E of \cite{NTV} and noticing that $L$ is central (and therefore, every irreducible character of $L$ is $G$-invariant), we have that
 
 $$\theta_C=e\varphi + p\Delta,$$ and
  $$(\theta_{p'})_{C\cap X}=f\varphi_{p'} + p\Delta',$$

  where $\Delta$ and $\Delta'$ are 0 or characters of $C$ and $C\cap X$, respectively. Recall that $\theta=\theta_{p'}\times\theta_{p}$ and $\varphi=\varphi_{p'}\times\theta_p$. Then
 
 $$\theta_C=(\theta_{p'})_{C\cap X}\times \theta_p =f(\varphi_{p'}+p\Delta')\times\theta_p=f\varphi\times pf(\Delta'\times \theta_p)$$ and hence, $e=f\equiv \varepsilon$ mod $p$, as wanted. Since $p$ is odd we have that $\varepsilon$ is uniquely determined, as wanted.
 \end{proof}

\begin{thm}\label{fullyGinvxi}
Suppose that $G$ is a finite group, $p$ is an odd prime, $K,L$ are normal subgroups of $G$ such that
$K/L$ is a $p'$-group. Suppose that $V=PK$ is normal in $G$, where $P \in \syl pG$.
Let $H=L\norm GP$ and write $K\cap H=C$.
Let $\theta \in \irr K$ be $G$-invariant, and let $\varphi \in \irr C$ be the relative $P$-Glauberman correspondent
of $\theta$. Let $\xi\in{\rm Irr}(L)$ under $\varphi$, and suppose that $\xi$ is $G$-invariant. Then there exists a bijection
$$f: {\rm Irr}(G|\theta) \rightarrow {\rm Irr}(H|\varphi)$$ such that
$\chi(1)_p=f(\chi)(1)_p$ for every $\chi \in \irr{G|\theta}$. Furthermore,   if $x \in P$ is such that $xL$ is picky in $G/L$, then 
$\chi(x)=\varepsilon f(\chi)(x)$, where $\varepsilon\in\{1,-1\}$ is the unique sign such that
 $[\theta_C,\varphi]\equiv \varepsilon$ mod $p$.
 \end{thm}
 \begin{proof}

By the theory of character triples (see Corollary 5.9 of \cite{N18}, for instance), let $(G^*,L^*,\xi^*)$ be a character triple isomorphic to $(G,L,\xi)$ with $L^*\subseteq{\rm\textbf{Z}}(G^*)$ and, with a slight abuse of notation, let $${}^*:{\rm Irr}(J|\xi)\rightarrow{\rm Irr}(J^*|\xi^*)$$ be the associated bijection of characters, where $L\leq J\leq G$ and $J/L\cong J^*/L^*$ under the corresponding isomorphism of groups. By the definition of character triple isomorphisms (see Definition 5.7(a) of \cite{N18}, for instance), $\chi\in{\rm Irr}(G|\theta)$ if and only if $\chi^*\in{\rm Irr}(G^*|\theta^*)$ (and the same with $\varphi$ and $\varphi^*$), therefore ${}^*$ also defines  bijections
$${}^*:{\rm Irr}(G|\theta)\rightarrow{\rm Irr}(G^*|\theta^*)$$ and

$${}^*:{\rm Irr}(H|\varphi)\rightarrow{\rm Irr}(H^*|\varphi^*).$$

Notice that since $xL$ is picky in $G/L$, then $x^*L^*$ is picky in $G^*/L^*$. 

By Theorem \ref{fullycentralcase}, we know that the result is true for $G^*$, $K^*$, $L^*$, $H^*$, $\theta^*$ and $\varphi^*$. Let $$f^*:{\rm Irr}(G^*|\theta^*)\rightarrow{\rm Irr}(H^*|\varphi^*)$$ be the corresponding bijection. Then the composition $$f:{\rm Irr}(G|\theta)\xrightarrow[]{{}^*}{\rm Irr}(G^*|\theta^*)\xrightarrow[]{f^*}{\rm Irr}(H^*|\varphi^*)\xrightarrow[]{{{}^{*}}^{-1}}{\rm Irr}(H|\varphi)$$ is a bijection. Since ratios of character degrees are preserved by character triple isomorphisms, we have that $\chi(1)_p=f(\chi)(1)_p$ for every $\chi\in{\rm Irr}(G|\theta)$, using that $\theta(1)_p=\varphi(1)_p$ by Corollary 11.29 of \cite{Is76}.

We need to check that $\chi(x)=\varepsilon f(\chi)(x)$ where $\varepsilon$ is the unique sign such that $[\theta_L,\varphi]\equiv \varepsilon$ mod $p$. Let $\tau$ be an extension of $\xi$ to $L\langle x\rangle$. By the proof of Lemma 5.17 of \cite{N18}, and using that the result is true for $G^*$,  we have that

$$\chi(x)=\frac{\tau(x)}{\tau^*(x)}\chi^*(x^*)=\frac{\tau(x)}{\tau^*(x)}\varepsilon f^*(\chi^*)(x^*),$$ where $\varepsilon\in\{1,-1\}$ is the unique sign such that $[\theta_L,\varphi]=[(\theta^*)_{L^*},\varphi^*]\equiv \varepsilon$ mod $p$. By the same argument,

$$f(\chi)(x)=\frac{\tau(x)}{\tau^*(x)}f^*(\chi^*)(x^*)$$ and we conclude that $$\chi(x)=\varepsilon f(\chi)(x),$$ as wanted. 
 \end{proof}
 
Next we want to remove the hypothesis of $\xi$ being $G$-invariant in Theorem \ref{fullyGinvxi}. To do so, first we need the following easy lemma.

\begin{lem}\label{lem:inductiondiamant} Let $G$ be a finite group, let $K\lhd G$ and let $H\leq G$ such that $G=KH$. Write $M=K\cap H$. Let $\delta\in{\rm Irr}(H)$, then
$$\delta^G(h)=\frac{1}{|M|}\sum_{k\in K\atop khk^{-1}\in H}\delta(khk^{-1}).$$
\end{lem}
\begin{proof} Write $\{h_1,\ldots,h_r\}$ for a complete set of coset representatives of $K$ in $G$, and notice that $r=|G:K|$. Then, for every $g\in G$, there exists a unique $i\in\{1,\ldots, r\}$ and a unique $k\in K$ such that $g=h_ik$. Notice that $ghg^{-1}\in H$ if, and only if, $khk^{-1}\in H$. Hence, we can write:
$$\delta^G(h)=\frac{1}{|H|}\sum_{g\in G\atop ghg^{-1}\in H}\delta(ghg^{-1})=\frac{|G:K|}{|H|}\sum_{k\in K\atop khk^{-1}\in H}\delta(khk^{-1})=\frac{1}{|M|}\sum_{k\in K\atop khk^{-1}\in H}\delta(khk^{-1})$$
\end{proof}

The following is the main result of this section and the key to proving Theorem A.

\begin{thm}\label{anticfully}
Suppose that $G$ is a finite group, $p$ is an odd prime, $K,L$ are normal subgroups of $G$ such that
$K/L$ is a $p'$-group. Suppose that $V=PK$ is normal in $G$, where $P \in \syl pG$.
Let $H=L\norm GP$ and write $K\cap H=C$.
Let $\theta \in \irr K$ be $G$-invariant, and let $\varphi \in \irr C$ be the relative $P$-Glauberman correspondent
of $\theta$. Then there exists a bijection
$$f: {\rm Irr}(G|\theta) \rightarrow {\rm Irr}(H|\varphi)$$ such that
$\chi(1)_p=f(\chi)(1)_p$ for every $\chi \in \irr{G|\theta}$. Furthermore,   if $x \in P$ is such that $xL$ is picky in $G/L$, then 
$\chi(x)=\varepsilon f(\chi)(x)$, where $\varepsilon\in\{1,-1\}$ is the unique sign such that
 $[\theta_C,\varphi]\equiv \varepsilon$ mod $p$.
 \end{thm}
 
 \smallskip
 
\begin{proof}~~ First we remark that the last assertion makes sense. Indeed, $\cent {K/L}P=C/L$ and by Theorem E of \cite{NTV}  we have that $[\theta_C,\varphi]\equiv \pm 1$ mod $p$.

We proceed by induction on $|G:L|$. By Lemma \ref{tom}, let $\xi\in{\rm Irr}(L)$ be $P$-invariant under $\varphi$ (and hence, under $\theta$). Suppose that $G_\xi< G$. Since $\theta$ is $G$-invariant, we have that $G=KG_\xi$. Notice that $G_\xi=K_\xi H_\xi$ and write $C_\xi=K_\xi\cap H_\xi$. Let $\theta_\xi\in{\rm Irr}(K_\xi|\xi)$ be the Clifford correspondent of $\theta$ over $\xi$, and let $\varphi_\xi\in{\rm Irr}(C_\xi|\xi)$ the Clifford correspondent of $\varphi$ over $\xi$. Notice that $\theta_\xi$ is $G_\xi$-invariant. Indeed, let $h\in G_\xi$, then $\theta_\xi^h$ lies under $\theta^h=\theta$ and over $\xi^h=\xi$, and by the Clifford correspondence (see Theorem 6.11 (c) of \cite{Is76}) we have that $\theta_\xi=\theta_\xi^h$. In the same way we have that $\varphi_\xi$ is $H_\xi$-invariant. We claim that $\varphi_\xi$ is the relative Glauberman correspondence of $\theta_\xi$. Indeed, 

$$\theta_{C_\xi}=(\theta_{K_\xi})_{C_\xi}=(\theta_\xi)_{C_\xi}+T_{C_\xi},$$ where the irreducible constituents of $T$ do not lie over $\xi$. Now write $\theta_C=e\varphi+p\Delta+\Xi$ as in Theorem \ref{rel}. Then

$$\theta_{C_\xi}=(\theta_C)_{C_\xi}=(e\varphi+p\Delta+\Xi)_{C_\xi}=e\varphi_\xi+p\Delta_{C_\xi}+\Xi_{C_\xi}+R,$$ where the irreducible constituents of $R$ do not lie over $\xi$. Since $\xi$ is $P$-invariant, no irreducible constituent of $\Xi$ lies over $\xi$ either. Therefore $$(\theta_\xi)_{C_\xi}=e\varphi_\xi+p\Delta'_{C_\xi},$$ where $\Delta'_{C_\xi}$ is a character of $C_\xi$ or zero, and the claim is proved.

 By induction there is a bijection satisfying all the desired properties:

$$\tilde f:{\rm Irr}(G_\xi|\theta_\xi)\rightarrow{\rm Irr}(H_\xi|\varphi_\xi).$$

We next claim that induction of characters defines a bijection

$${\rm Irr}(G_\xi|\theta_\xi)\rightarrow{\rm Irr}(G|\theta).$$

First, if $\psi\in {\rm Irr}(G_\xi|\theta_\xi)\subseteq{\rm Irr}(G_\xi|\xi)$, we have that $\psi^G\in{\rm Irr}(G|\xi)$. Now by Problem 5.2 of \cite{Is76}, we have that $[(\psi^G)_K,\theta]=[(\psi_{K_\xi})^K,\theta]\neq 0$ (we are using that  $\psi$ lies over $\theta_\xi$ and $\theta_\xi^K=\theta$). We conclude that $\psi^G\in{\rm Irr}(G|\theta)$ and the map is well-defined. The injectivity follows from the injectivity of the Clifford correspondence. Finally, let $\chi\in{\rm Irr}(G|\theta)\subseteq{\rm Irr}(G|\xi)$ and let $\psi\in{\rm Irr}(G_\xi|\xi)$ such that $\psi^G=\chi$. Since $\theta$ is $G$-invariant we have that $(\psi^G)_K=\chi_K=f\theta$ for some $f$. Since $\psi$ lies over $\xi$, every irreducible constituent of $\psi_{K_\xi}$ lies over $\xi$. Since $(\psi^G)_K=(\psi_{K_\xi})^K$ we conclude that $\psi$ lies over $\theta_\xi$ by Theorem 6.11 (c) of \cite{Is76}, as wanted. In the same way, induction defines a bijection

$${\rm Irr}(H_\xi|\varphi_\xi)\rightarrow{\rm Irr}(H|\varphi).$$ 

Composing these bijections we obtain a bijection

$$f:{\rm Irr}(G|\theta)\rightarrow{\rm Irr}(H|\varphi)$$
 by sending $\chi\in{\rm Irr}(G|\theta)$ to $f(\chi)=\tilde{f}(\psi)^H$, where $\chi=\psi^G$. Now our goal is to prove that, if $\psi\in{\rm Irr}(G_\xi|\theta_\xi)$, then $$\chi(x)=\psi^G(x)=\varepsilon\tilde f(\psi)^H(x)=\varepsilon f(x).$$

Recall that $G=KG_\xi$ and $K_\xi=K\cap G_\xi$. By Lemma \ref{lem:inductiondiamant} we have that

$$\psi^G(x)=\frac{1}{|K_\xi|}\sum_{k\in K\atop kxk^{-1}\in G_\xi}\psi(kxk^{-1}).$$ 

Using that $H=CH_\xi$, writing $C_\xi=C\cap H_\xi$, and applying again Lemma \ref{lem:inductiondiamant} we have that

$$\tilde f(\psi)^H(x)=\frac{1}{|C_\xi|}\sum_{c\in C\atop cxc^{-1}\in H_\xi}\tilde f(\psi)(cxc^{-1})=\frac{1}{|C_\xi|}\sum_{c\in C}\tilde f(\psi)(cxc^{-1}),$$ where the last equality follows from the fact that, if $c\in C$, then $x^cL=xL$ and hence $x^c=xl\in H_\xi$ for every $c\in C$. Now write $\Delta=\{k\in K\>|\> kxk^{-1}\in G_\xi\}$, so we have that

$$\psi^G(x)=\frac{1}{|K_\xi|}\sum_{k\in \Delta}\psi(kxk^{-1}).$$ 

Write $C=C_\xi c_1\sqcup C_\xi c_2\sqcup\ldots\sqcup C_\xi c_s$ as a disjoint union, where $c_i\in C$. We claim that

$$\Delta=K_\xi c_1\sqcup\ldots\sqcup K_\xi c_s.$$ Indeed, let $k\in \Delta$, so that $kxk^{-1}\in G_\xi$. Since $G_\xi/K_\xi$ has a normal Sylow subgroup, say $PK_\xi/K_\xi$, we have that $kxk^{-1}\in K_\xi P$ and since $PL/L$ is a Sylow $p$-subgroup of $K_\xi P/L$, there exists $v\in K_\xi$ such that $vkxk^{-1}v^{-1}\in PL$. Since $xL$ is picky in $G/L$, we have that $vkL\in\norm {G/L}{PL/L}$ and hence $vk\in \norm G P L=H$. But then $vk\in H\cap K=C$. Write $vk=c$ for some $c\in C$ and write $c=c'c_i$ for some $c'\in C_\xi$. Then $k=v^{-1}c'c_i\in K_\xi c_i$. The reverse inclusion is obvious. Finally, suppose that $k\in \Delta$ and write $k=k_ic_i=k_jc_j$ with $k_i,k_j\in K_\xi$. Then $c_jc_i^{-1}=k_j^{-1}k_i\in K_\xi\cap C=C_\xi$ and we can write $c_j=cc_i$ for some $c\in C_\xi$. We conclude that $i=j$, and the claim is proved. Now,

$$\psi^G(x)=\frac{1}{|K_\xi|}\sum_{j=1}^s\sum_{k\in K_\xi}\psi(kc_jxc_j^{-1}k^{-1})= \sum_{j=1}^s\psi(c_jxc_j^{-1})=\frac{\varepsilon}{|C_\xi|}\sum_{c\in C}\tilde f(\psi)(cxc^{-1}),$$ where we have used in the last equality that $c_jxc_j^{-1}\in H$ and $c_jxc_j^{-1}L$ is also picky in $G/L$, $H=L\norm G P=L\norm G P^c$ for every $c\in C$ and $\varphi$ is also the relative $P^c$-Glauberman correspondent of $\theta$ for every $c\in C$. Therefore $\psi(c_jxc_j^{-1})=\varepsilon \tilde f(\psi)(c_jxc_j^{-1})$. We conclude that 

$$\psi^G(x)=\varepsilon \tilde f(\psi)^H(x),$$ as wanted. This proves the result in this case. Hence we may assume that $\xi$ is $G$-invariant, and we apply Theorem \ref{fullyGinvxi} to conclude.

\medskip

\end{proof}

 \section{Proof of Theorem A}
 
 We start with a very elementary lemma.
 
  \begin{lem}
\label{subg}
Let $G$ be a finite group and let $x\in H\leq G$. If $x$ is picky in $G$, then $x$ is picky in $H$.
\end{lem}

\begin{proof}
Let $Q$ be a Sylow $p$-subgroup of $H$ such that $x\in Q$. Since $x$ is picky as an element of $G$, there exists a unique Sylow $p$-subgroup $P$ of $G$ containing $Q$. Therefore, $Q=P\cap H$ and it follows that $Q$ is unique.
\end{proof}

\begin{thm}\label{mainpodd}
Suppose that $G$ is a finite group, $p$ is an odd prime, $K,L$ are normal subgroups of $G$ such that
$K/L$ is a $p'$-group. Suppose that $V=PK$ is normal in $G$, where $P \in \syl pG$.
Let $H=L\norm GP$ and write $K\cap H=C$.
Let $xL\in PL/L$ be a picky $p$-element in $G/L$. Let $\theta \in \irr K$ be $x$-invariant, and let $\varphi \in \irr C$
be the relative $\langle x \rangle$-Glauberman correspondent of $\theta$.  Then there exists a bijection
$$^*: {\rm Irr}(G|\theta) \rightarrow  {\rm Irr}(H|\varphi)$$ and a sign $\epsilon_x$ such that
$\chi(x)=\epsilon_x \chi^*(x)$ and $\chi(1)_p=\chi^*(1)_p$ for every $\chi \in \irr{G|\theta}$.
\end{thm}

\begin{proof} Since $xL$ is picky in $G/L$, we have that $\cent{K/L} x=C/L=\cent {K/L}P$
by Lemma \ref{elem2}. Write $T=G_\theta$.  By Lemma \ref{inv2}, we know that $H_\theta=H_\varphi$. 
Also, $xL$ is picky in $T/L$. Let $P_\theta=P\cap T$, and notice that $C/L=\cent {K/L}{P_\theta}$. Then $\varphi$ is the relative $P_\theta$-Glauberman correspondent of $\theta$ and by Theorem \ref{anticfully},
there is a bijection
$$\tilde{}: \irr{T|\theta} \rightarrow \irr{H_\varphi|\varphi}$$ such that $\psi(1)_p=\tilde\psi(1)_p$, and $\psi(y)=\epsilon \tilde\psi(y)$,
where $[\theta_C, \varphi] \equiv \epsilon$ mod $p$, for all $\psi \in \irr{T|\theta}$, whenever $yL$
is a picky $p$-element of $T/L$.
Now, we define $$(\psi^G)^*=\tilde\psi^H$$
for $\psi \in \irr{T|\theta}$. 
We only need to check that 
$$\psi^G(x)=\epsilon \tilde \psi^H(x) \, .$$
Suppose that $h \in H$ is such that $y=hxh^{-1} \in H_\varphi$.
Since $xL$ is picky in $G/L$,
we have that $yL$ is picky in $G/L$ and hence $yL$ is picky in every subgroup that contains it
by Lemma \ref{subg}.
We conclude that
$$\psi^G(x)=(\psi_{H_\varphi})^H (x)={1 \over |H_\varphi|} \sum_{h \in H \atop hxh^{-1} \in H_\varphi} \psi(hxh^{-1})=
{1 \over |H_\varphi|} \sum_{h \in H \atop hxh^{-1} \in H_\varphi} \epsilon \tilde \psi(hxh^{-1})$$
$$=\epsilon \tilde\psi^H(x), $$
as wanted. 
\end{proof}

\begin{lem}
\label{xinv}
Let $G$ be a finite group, let $P\in{\rm Syl}_p(G)$ and let $x\in P$. Let $N$ be a normal subgroup of $G$. If $\chi\in{\rm Irr}^x(G)$ then there exists $\theta\in\irr N$ lying under $\chi$ that is $\langle x\rangle$-invariant. Furthermore if $x$ is picky, then  any two such characters are $\norm G P$-conjugate.
\end{lem}

\begin{proof}
Let $\theta\in\irr N$ lying under $\chi$. Put $T=G_{\theta}$. Let $\psi\in\irr T$ be the Clifford correspondent of $\chi$ over $\theta$. Therefore,
$$
0\neq\chi(x)=\frac{1}{|T|}\sum_{g\in G}\psi^o(gxg^{-1})
$$
so there exists $g\in G$ such that $gxg^{-1}\in T=G_{\theta}$. It follows that $x$ fixes $\theta^{g}$.

Now, assume that $x$ is picky and let $\theta_1,\theta_2\in\irr N$ be $\langle x\rangle$-invariant lying under $\chi$. There exists $g\in G$ such that $\theta_1=\theta_2^g$. Since $x\in G_{\theta_i}$ for $i=1,2$, $x$ is picky as an element of $G_{\theta_i}$  by Lemma \ref{subg}. Let $Q_i\in{\rm Syl}_p(G_{\theta_i})$ be the unique Sylow subgroup of $G_{\theta_i}$ that contains $x$. We have that $Q_i=P\cap G_{\theta_i}$  for $i=1,2$. Thus $Q_1^{g^{-1}},Q_2\subseteq G_{\theta_2}$ are both Sylow subgroups, whence there exists $h\in G_{\theta_2}$ such that $Q_1^{g^{-1}h}=Q_2$.  Hence $x$ belongs to both $P$ and $P^{g^{-1}h}$. Since $x$ is picky, we have that $g^{-1}h\in\norm G P$ and therefore $g=hy$ for some $y\in\norm G P$. Now $\theta_1=\theta_2^g=\theta_2^{y}$ and it follows that $\theta_1$ and $\theta_2$ are $\norm G P$-conjugate, as wanted.
\end{proof}

We need one more lemma.
\begin{lem}
Let $G$ be a finite group and let $N$ be a normal subgroup of $G$.
Let $x \in G$ be a picky  $p$-element in $G$.  Then $xN$ is a picky $p$-element in $G/N$.
\end{lem}

\begin{proof}
Let $P\in{\rm Syl}_p(G)$ be the unique Sylow subgroup containing $x$. We will prove that $PN/N$ is the unique Sylow $p$-subgroup of $G/N$ containing $xN$. Suppose that $xN\in QN/N$ where $Q\in{\rm Syl}_p(G)$. Then $x\in QN$ and since $QN$ contains a Sylow $p$-subgroup of $G$, it follows that $P\subseteq QN$. Then $P$ and $Q$ are Sylow $p$-subgroups of $QN$ and hence, there exists $n\in N$ such that $P=Q^n$. We conclude that $PN=Q^nN=QN$, as wanted.
\end{proof}

We are ready to prove Theorem A.

\medskip

\begin{thm}\label{thmApodd}
Let $G$ be a finite $p$-solvable group, where $p$ is an odd prime, and $x \in G$ is a $p$-element which lies in a unique
Sylow $p$-subgroup $P$ of $G$. Then there is a bijection $$^*: {\rm Irr}^{x}(G) \rightarrow {\rm Irr}^{x}(\norm GP)$$
and a sign $\epsilon_x$ such that $\chi(x)=\epsilon_x \chi^*(x)$ and $\chi(1)_p=\chi^*(1)_p$ for all 
$\chi \in {\rm Irr}^{x}(G)$.
\end{thm}

\begin{proof}
We argue by induction on $|G|$. If $P \nor G$, there is nothing to prove. 
Let $K=\Oh{p'p}{G}$ and let $K/L$ be a chief factor of $G$.  Notice that $KP\lhd G$ and then $G=K\norm G P$. Let $H=L\norm GP$ and notice that $KH=G$ and $K\cap H=C$, where $C/L=\cent{K/L}P$. Also notice that $H<G$, since otherwise $LP\lhd G$ and then $KP={\rm \textbf{O}}^{p'}(G)\subseteq LP$ and $K=L$, a contradiction.

By induction, there is a bijection
$f:{\rm Irr}^{x}(H) \rightarrow  {\rm Irr}^{x}(\norm GP)$ and a sign $\delta_x$
such that $\chi(x)=\delta_x f(\chi)(x)$ and $\chi(1)_p=f(\chi)(1)_p$ for all 
$\chi \in {\rm Irr}^{x}(L\norm GP)$.

Let $\Delta$ be a complete set of $H$-representatives of the $x$-invariant characters of $\irr K$.
Let $\Xi$ be consisting of all relative $\langle x\rangle$-Glauberman correspondents of $\theta$.
By Lemma \ref{xinv} we have that
 $${\rm Irr}^{x}(G)=\bigcup_{\theta \in \Delta} {\rm Irr}^{x}(G|\theta)$$ and 
  $${\rm Irr}^{x}(H)=\bigcup_{\varphi \in \Xi} {\rm Irr}^{x}(H|\varphi) \, $$
 where  ${\rm Irr}^{x}(G|\theta)={\rm Irr}^{x}(G) \cap \irr{G|\theta}$.
 The proof of the theorem is concluded by using Theorem \ref{mainpodd}.
\end{proof}

\medskip

In our proof of Theorem A, we used oddness in two key places. In Theorem \ref{anticfully}, we saw that
the sign associated to picky elements was universal (that is, not depending on the particular $p$-element 
in consideration).
Also, when we applied Turull's result, we needed that $p$ is odd, in order to control certain character triple isomorphisms.
The universal sign is not going to hold when $p=2$.  The smallest example is associated with ${\rm Sp}(6,3)$.
Indeed, suppose that $K$ is an extraspecial 3-group of order $3^7$ and exponent 3. Let $L=\zent K$. 
Suppose that $1 \ne \varphi \in \irr Z$ and write $\varphi^K=e\theta$, where $e=27$ and $\theta \in \irr K$.
Let $H={\rm Sp}(6,3)$ and let $\Psi$ be any one of the two Weil characters of degree 27 of $H$. 
It can be shown that there is a 2-group $A={\sf C}_4 \times {\sf C}_4 \le H$ having elements $x, y \in A$
of order 4 such that $\cent{K}x=\cent Ky=L$ and $\Psi(x)=1$ and $\Psi(y)=-1$.  Let $G=KA$ be the semidirect product,
where $[A,L]=1$, and let $H=LA$. 
Let $\hat\theta \in \irr G$ be the canonical extension of $\theta$ to $G$.  Using the theory in \cite{Is73},
or using GAP, one can check that $\hat\theta(x)=\Psi(x)=1$ and $\hat\theta(y)=\Psi(y)=-1$. Hence, the sign $\epsilon$  in Theorem \ref{anticfully} changes for distinct elements, even if they have the same order.  When $p$ is odd, however,
 it can be shown that if $x$ is a $p$-element in $H$ with $\cent{K/L}x=1$ and $(|K/L|,p)=1$, then $\Psi(x)=\epsilon$, where $\epsilon$ is the only sign such that $|K:L|^{1/2} \equiv \epsilon$ mod $p$.
 This is of course consistent with Theorem \ref{anticfully}.
Moreover, if $o(x)=2$, then $\Psi(x)=\epsilon$ where $\epsilon \equiv |K:L|^{1/2} $ mod 4, and this again
can be used to prove Theorem A for
elements of order 2.  At the time of this writing, our techniques do not allow us to prove Theorem A for 2-elements of greater order.

\subsection*{Acknowledgements} We thank Gunter Malle for useful conversations on this subject and for his careful reading of this manuscript.

\end{document}